\documentclass[11pt,a4paper,leqno]{amsart}
\usepackage[T1]{fontenc}
\pdfoutput=1

\usepackage{anyfontsize}
\usepackage{amssymb,amsfonts,mathtools}   
\usepackage{graphicx}
\usepackage{placeins}
\usepackage{enumerate}
\usepackage{tikz}
\usepackage{setspace}
\usepackage{xcolor}
\usepackage{hyperref}
\usepackage{theoremref}

\usepackage[
  margin=1.5cm,
  includefoot,
  footskip=30pt,
]{geometry}

\usepackage[
backend=biber,
style=alphabetic,
maxbibnames=99, 
]{biblatex}
\newtheorem{thm}{Theorem}[section]
 
\newtheorem{lem}[thm]{Lemma} 
\newtheorem{prop}[thm]{Proposition}

\theoremstyle{definition}

\theoremstyle{remark}

\numberwithin{equation}{section}

\newcommand{\R}{\mathbb{R}}

\newcommand{\N}{\mathbb{N}}

\newcommand{\Sph}{\mathbb{S}}

\newcommand{\dist}{\operatorname{dist}}

\newcommand{\dd}{\mathop{}\!\mathrm d} 
\renewcommand{\H}{\mathcal H } 

\newcommand{\Per}{\operatorname{Per}}

\newcommand{\reach}{\operatorname{reach}}
\usepackage{newpxtext}
\usepackage[varbb]{newpxmath} 
\renewcommand\mathcal[1]{\text{\usefont{OMS}{cmsy}{m}{n}#1}}
\begin{document}

\title[The nonlocal tube formula]{On the Weyl tube formula for nonlocal perimeters}

\author[S. Dipierro]{Serena Dipierro}
\address{The University of Western Australia, 35 Stirling Highway, PERTH, WA, 6009, Australia}
\author[J. Thompson]{Jack Thompson}
\address{KTH Mathematics, Lindstedsv\"agen 25, 100 44 Stockholm, Sweden}
\author[E. Valdinoci]{Enrico Valdinoci}
\address{The University of Western Australia, 35 Stirling Highway, PERTH, WA, 6009, Australia}
\email{serena.dipierro@uwa.edu.au}
\email{enrico.valdinoci@uwa.edu.au}
\email{jackth@kth.se}


\date{\today}

\dedicatory{}

\begin{abstract}
 We study the asymptotic behaviour of the nonlocal perimeter of tubular neighbourhoods. For a closed $C^{1,1}$ submanifold of Euclidean space, we show that a suitable rescaling of the nonlocal perimeter of its tube converges to an explicit constant times the Hausdorff measure of the submanifold. The constant is the fractional perimeter of the unit ball in the codimension of the submanifold. For finite point sets, we obtain an exact tube formula in terms of a generalised hypergeometric function. In contrast with the classical Weyl formula, this expansion has infinitely many terms, contains both noninteger and integer powers of the tubular radius, and has coefficients that depend on the relative positions of the points rather than only on their intrinsic geometry. We also establish a point-set expansion for general analytic kernels.
\end{abstract}

\maketitle

\section{Introduction}

Weyl's tube formula relates the volume of a tubular neighbourhood of a compact submanifold to its geometry. In Euclidean space, the volume is a polynomial in the tubular radius, with coefficients determined by intrinsic curvature invariants of the submanifold \cite{weyl_volume_1939,gray_tubes_2004}. This formula is closely related to Steiner's formula for convex bodies both of which fit within Federer's theory of sets of positive reach and curvature measures \cite{federer_curvature_1959,schneider_convex_2014,rataj_curvature_2019}. In this note, we investigate Weyl's formula when the volume is replaced by the nonlocal perimeter. In particular, we address the leading-order behaviour of the fractional perimeter of a shrinking tube and the structure of its full expansion in the tubular radius for zero-dimensional submanifolds.

Fractional perimeters, introduced in \cite{caffarelli_nonlocal_2010}, arise in geometric variational problems and phase transition models involving long-range interactions. The nonlocal perimeter has seen consider interest over the last twenty year and has been developed in several directions including isoperimetric-type inequalities such as in \cite{cabre_fractional_2023,figalli_isoperimetry_2015,brasco_fractional_2014,frank_hardy-lieb-thirring_2008}, and the regularity theory for critical/stable/minimising sets, see \cite{dipierro_nonlocal_2018} for a survey on known results. The nonlocal perimeter also gives a natural interpolation between volume and the classical perimeter both of which can be recovered in the limit as the fractional parameter approaches either endpoint of its range of admissible values. From this viewpoint, the nonlocal perimeter can be viewed as a generalisation of both the volume and the perimeter. 

\begin{figure}[htbp]
    \centering 
    \includegraphics[width=0.45\textwidth]{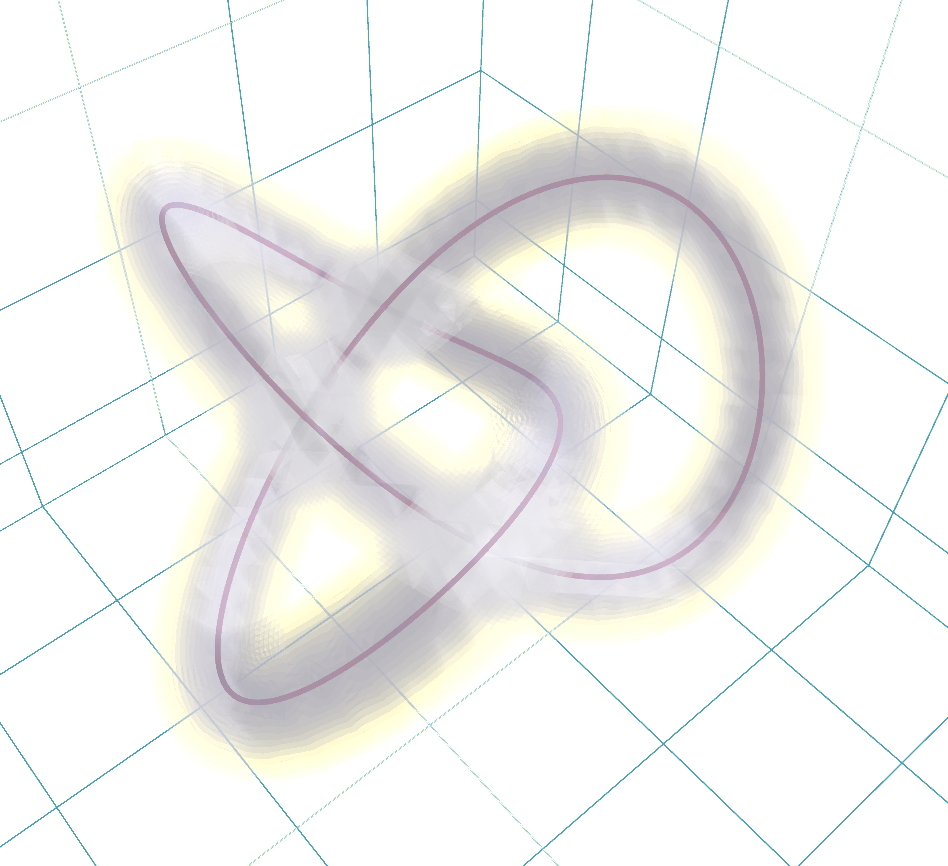} \caption{}
    \label{fig:my}
\end{figure}

For a set $E\subset\R^N$ and $\rho\geq0$, write
\begin{align*}
    E_\rho:=\{x\in\R^N \text{ s.t. } \dist(x,E)\leq\rho\}.
\end{align*}
The \emph{reach} of a closed set $E$ is the supremum of the radii for which every point at distance strictly less than that radius from $E$ has a unique nearest point in $E$. We first recall the classical formula for a closed $C^2$ submanifold $M\subset\R^N$, of dimension $m$ and codimension $n=N-m\geq1$
(see Figure~\ref{fig:my} for a sketch
of a tubolar neighborhood of a closed curve in~$\R^3$). In the notation of \cite{weyl_volume_1939,federer_curvature_1959}, it reads
\begin{align}
    \H^N(M_\rho)
    =\sum_{k=n}^N\H^k(B_1^k)\,\mathbf C_{N-k}(M)\rho^k,
    \qquad 0<\rho<\reach(M),\label{eq:Weyl}
\end{align}
where $\mathbf C_j(M)$ denotes the $j$-th total curvature. These coefficients are intrinsic since they can be expressed through the Riemannian curvature tensor of $M$, see \cite[Remark 5.21]{federer_curvature_1959}. In particular, $\mathbf C_m(M)=\H^m(M)$ and $\mathbf C_j(M)=0$ when $m-j$ is odd. Thus, we have
\begin{align}
    \lim_{\rho\to0^+}\rho^{-n}\H^N(M_\rho)
    =\H^n(B_1^n)\,\H^m(M).\label{eq:Weyl-leadingterm}
\end{align}
Differentiating the tube-volume formula and using the coarea formula gives the corresponding classical perimeter formula
\begin{align}
    \H^{N-1}(\partial M_\rho)
    =\sum_{k=n}^N\Per^k(B_1^k)\,\mathbf C_{N-k}(M)\rho^{k-1},
    \qquad 0<\rho<\reach(M),\label{eq:Weyl-perimeter}
\end{align}
where $\Per^k(B_1^k)=k\H^k(B_1^k)$. Hence, we also have
\begin{align}
    \lim_{\rho\to0^+}\rho^{1-n}\H^{N-1}(\partial M_\rho)
    =\Per^n(B_1^n)\,\H^m(M).\label{eq:Weyl-perimeter-leadingterm}
\end{align}

The goal of this paper is to understand the
nonlocal counterpart of this classical theory.
To this end,
given $s\in(0,1)$ and a measurable set $E\subset\R^N$, we define its $s$-perimeter by
\begin{align*}
    \Per_s^N(E)
    :=\gamma(N,s)\int_{\R^N\setminus E}\int_E
    \frac{\dd y\,\dd x}{\vert x-y\vert^{N+s}},
    \qquad
    \gamma(N,s):=\pi^{-N/2}\Gamma\bigg(\frac{N+s}{2}\bigg).
\end{align*}
The superscript records the ambient dimension, which will change in the leading-order limit. Our first result is the nonlocal counterpart of \eqref{eq:Weyl-leadingterm} and \eqref{eq:Weyl-perimeter-leadingterm}.

\begin{thm} \thlabel{th:main}
Let $M\subset\R^N$ be a closed $C^{1,1}$ submanifold of dimension $1\leq m\leq N-1$. If $n=N-m$ then
\begin{align}
    \lim_{\rho\to0^+}\rho^{s-n}\Per_s^N(M_\rho)
    =\Per_s^n(B_1^n)\,\H^m(M).\label{eq:thm-main}
\end{align}
\end{thm}

The power $\rho^{n-s}$ and the appearance of the fractional perimeter in dimension $n$ can be explained geometrically since, at small scales, the manifold~$M$ resembles its tangent plane and its tubular neighbourhood resembles a product of that plane with a ball in the $n$ normal directions. Integrating the fractional kernel along the $m$ tangential directions produces precisely the fractional kernel in the normal space. The choice of the normalisation constant \(\gamma(N,s)\) plays an important role at this stage since it guarantees that new normalisation constant obtained for fractional kernel in the normal space depends only on \(n\) and \(s\), not \(m\). It is interesting to find that, although fractional perimeter is nonlocal, its leading-order tube asymptotics depend only on the Hausdorff measure of $M$, which is, instead, a local quantity. 

This is interesting, since
it suggests the idea that tubes tend
to reconstruct the local properties of the set
out of the fractional perimeter,
the nonlocal contribution
remaining only in the normalizing constant
in~\eqref{eq:thm-main}.

In fact,
one possible motivation for studying the limit
in~\eqref{eq:thm-main} is the possibility of obtaining a nonlocal mass that generalises the usual nonlocal perimeter to manifolds of higher-codimensional. Several notions of nonlocal measure or mass in higher codimension have recently been proposed \cite{cicalese_notion_2025,caselli_nonlocal_2024,mihaila_definition_2025}; see also \cite[Section 5]{serra_nonlocal_2024}. Theorem~\ref{th:main}, however, shows that the particular leading-order renormalisation considered here yields only a multiple of Hausdorff measure.

The constant in \eqref{eq:thm-main} also has the expected endpoint behaviour. For a bounded set $E\subset\R^d$ with sufficiently smooth boundary, the standard limits with our normalisation are
\begin{align*}
    \lim_{s\to0^+}s\Per_s^N(E)=2\H^N(E),
    \qquad
    \lim_{s\to1^-}(1-s)\Per_s^N(E)
    =\pi^{-1/2}\H^{N-1}(\partial E);
\end{align*}
see \cite{dipierro_asymptotics_2013,caffarelli_regularity_2013,ambrosio_gamma-convergence_2011}. Formally taking the limit \(s\to 0^+\) and \(s\to 1^-\) with \(E=M_\rho\) in \thref{th:main}, we recover~\eqref{eq:Weyl-leadingterm} and~\eqref{eq:Weyl-perimeter-leadingterm} respectively.\newline


Yet, beyond the leading term, the nonlocal picture can be substantially different from Weyl's classical formula. In particular, the full expansion need not be a finite polynomial in \(\rho\), different types of powers of \(\rho\) may occur in the same expansion, and the coefficients need not be determined by intrinsic geometric invariants of the underlying manifold. These phenomena are already visible for the simplest submanifolds, namely a set of finitely many points. In the following result, we establish the full nonlocal Weyl tube formula for zero dimensional compact manifolds, i.e. a finite set of disjoint points.

\begin{thm}\thlabel{cor:pointsets}
Let $P\geq2$ be an integer and let $M=\{a_1,\dots,a_P\}$ consist of pairwise distinct points of $\R^N$. Then, for $0<\rho<\frac12\min_{i\neq j}\vert a_i-a_j\vert$,
\begin{align}\label{eq:cor:pointsets}
    \Per_s^N(M_\rho)
    ={}&\Per_s^N(B_1^N)\H^0(M)\rho^{N-s}\\
    &-\gamma(N,s)\big(\H^N(B_1^N)\big)^2\rho^{2N}
    \sum_{\substack{1\leq i,j\leq P\\i\neq j}}
    \vert a_i-a_j\vert^{-(N+s)}
    {}_3F_2\!\left(\begin{matrix}
    \frac{N+s}{2},\ 1+\frac{s}{2},\ \frac{N+1}{2}\\
    \frac N2+1,\ N+1
    \end{matrix};\frac{4\rho^2}{\vert a_i-a_j\vert^2}\right).\nonumber
\end{align}
Here ${}_3F_2$ denotes the generalised hypergeometric function.
\end{thm}

The hypergeometric series is convergent under the stated restriction on $\rho$ and its two leading order terms are given by
\begin{align*}
    \Per_s^N(M_\rho)
    =P\Per_s^N(B_1^N)\rho^{N-s}
    -\gamma(N,s)\big(\H^N(B_1^N)\big)^2
    \bigg(\sum_{\substack{1\leq i,j\leq P\\i\neq j}}\frac 1{\vert a_i-a_j\vert^{N+s}} \bigg)\rho^{2N}
    +O(\rho^{2N+2})
\end{align*}
as $\rho\to0^+$. We see that the expansion has a noninteger leading power $N-s$ and integer powers, in contrast with the finite polynomial appearing in the classical tube formula. Furthermore, from the definition of the hypergeometric function~\eqref{eq:hypergeometric}, we see that the Maclaurin series of~\eqref{eq:cor:pointsets} gives infinitely many terms with integer powers \(2N+2k\), with~\(k\in\{ 0,1,2,\dots\}\).

Observe that the first term in~\eqref{eq:cor:pointsets} depends only on the number of points \(P\)
(that is~$\H^0(M)$). This term can be viewed as the direct nonlocal analogue of the classical tube formula for a point set since is it exactly the \(s\)-perimeter of a \(\rho\)-tubular neighbourhood (i.e. a ball of radius \(\rho\)) of each individual point summed over the set of points without taking into account any nonlocal interactions. The remaining term in~\eqref{eq:cor:pointsets} involves their pairwise distances and serves as a nonlocal interaction term. This term is truely nonlocal in the sense that it does not appear in the classical tube formula of a point set. For a single point, by comparison, there is no interaction term and $\Per_s^N(M_\rho)=\Per_s^N(B_1^N)\rho^{N-s}$ for all $\rho>0$.

This examples makes it clear that the coefficients in the nonlocal tube formula need not be geometric invariants. Indeed, any two finite zero-dimensional Riemannian manifolds with the same number of points are intrinsically isometric, while their embeddings in $\R^N$ can have different pairwise distances. Thus,
differently from the classical case,
the interaction coefficients therefore are not \emph{intrinsic} geometric invariants, even though pairwise distances are, of course, invariants under ambient Euclidean isometries. In Section~\ref{sec:points}, we also establish an expansion of the point-set interaction for general analytic kernels, see \thref{CC0og8h1}).

\subsection*{Organisation of paper}

In Section~\ref{sec:proof-main}, we give the proof of the main theorem, \thref{th:main}. In Section~\ref{sec:points}, we prove \thref{cor:pointsets} as well as a full tube formula for nonlocal perimeters with general kernels, \thref{CC0og8h1}.

\subsection*{Acknowledgments}
The first author was supported by the Australian Research Council Future Fellowship FT230100333
{\em New perspectives on nonlocal equations}.
The third author was supported by the Australian Laureate Fellowship FL190100081 {\em Minimal
surfaces, free boundaries and partial differential equations}.


\section{Proof of Theorem \ref{th:main}} \label{sec:proof-main}

The purpose of this section is to prove \thref{th:main}. We write
\begin{align*}
    N^1M = \{(x,\nu):x\in M,\ \nu\in N_xM,\ \vert \nu\vert=1\}
\end{align*}
for the unit normal bundle and \(N_x^1M\) for its fibre over \(x\). The proof uses the following proposition, which is a formulation of the classical Weyl formula for closed \(C^{1,1}\) submanifolds that is suitable for our needs.

\begin{prop} \thlabel{4gkQqROo} Let \(M\subset \R^N\) be a closed \(C^{1,1}\) submanifold of dimension \(m\) with \(1\leq m \leq N-1\) and let \(n=N-m\). Then there exist functions \(c_i \in L^\infty(N^1M)\) (depending on \(M\)) for \(i=0,\dots,m\) such that for all non-negative measurable functions \(h:\R^N \to \R\) and all \(0<\rho < \reach(M)\), we have
\begin{align*}
    \int_{M_\rho} h(z) \dd \H^N_z &= \sum_{i=0}^m\int_0^\rho \int_M \int_{N_x^1M} h(x+t\nu) t^{n-1+i} c_i(x,\nu) \dd \H^{n-1}_\nu \dd \H^m_x \dd t .
\end{align*}
Furthermore, \(c_0=1\).
\end{prop}

\thref{4gkQqROo} follows from the change-of-variables formula in tubular coordinates and the polynomial structure of its Jacobian as in \cite[Lemma 5.2 and Corollary 5.3]{federer_curvature_1959}. Next, we present the following technical result. 

\begin{prop} \thlabel{HmhFO2hM} Let \(M\subset \R^N\) be a closed \(C^{1,1}\) manifold of dimension \(m\) with \(1\leq m\leq N-1\) and let \(n=N-m\). There exist \(R,\delta>0\), depending only on \(M\), such that the following holds. Suppose
\begin{align*}
    0<\rho<\delta, \qquad x\in M, \qquad \nu\in N_x^1M, \qquad \tau\in[0,1], \qquad t\in[1,R/\rho], \qquad j\in\{0,\dots,m\}.
\end{align*}
Define
\begin{align*}
 I_j^{(\rho)}(x,\nu,t,\tau) := \int_M \int_{N_y^1M} \frac{c_j(y,\eta)}{\vert x-y+\rho t\nu-\rho\tau\eta\vert^{N+s}}\dd \H^{n-1}_\eta\dd \H^m_y .
\end{align*}
Then, for \((t,\tau)\neq(1,1)\),
\begin{align}
    \big \vert I_j^{(\rho)}(x,\nu,t,\tau)\big \vert
    \leq C\bigg(\frac{1}{\rho^{n+s}}\int_{N_x^1M}\frac{\dd \H^{n-1}_\eta}{\vert t\nu-\tau\eta\vert^{n+s}}+1\bigg), \label{7qouPfas}
\end{align}
where \(C>0\) depends only on \(m\), \(n\), \(s\), and \(M\). Furthermore, for every \((t,\tau)\neq(1,1)\),
\begin{align}
    \lim_{\rho\to0^+}\rho^{n+s}I_0^{(\rho)}(x,\nu,t,\tau)
    &=\frac{\gamma(n,s)}{\gamma(N,s)}\int_{N_x^1M}\frac{\dd \H^{n-1}_\eta}{\vert t\nu-\tau\eta\vert^{n+s}}. \label{QylvaPee}
\end{align}
\end{prop}

\begin{proof}
By translating and rotating, we may assume without loss of generality that \(x=0\in M\), that \(T_0M=\R^m\times\{0_n\}\), and that \(\nu=(0,\omega)\) for some \(\omega\in\Sph^{n-1}\). Due to the regularity and compactness assumptions on \(M\), there exists \(\varepsilon=\varepsilon(M)>0\) and \(u\in C^{1,1}_c(\R^m;\R^n)\) satisfying \(u(0)=0\), \(\nabla u(0)=0\), and \(\|\nabla u\|_{C^{0,1}(\R^m;\R^n)}\leq C(M)\) such that
\begin{align}
    M\cap B_\varepsilon = \{(y,u(y))\text{ s.t. }y\in\R^m\}\cap B_\varepsilon. \label{17EZvdE8}
\end{align}
Also, there exist \(\nu_k\in C^{0,1}_c(\R^m;\R^N)\), with~\(k=1,\dots,n\), such that \(\{\nu_k(y)\}_{k=1}^n\) is an orthonormal basis for \(N_{(y,u(y))}M\) for all \(y\in U_\varepsilon\), where \(U_\varepsilon\subset\R^m\) is the preimage of \(M\cap B_\varepsilon\) under the map \(y\mapsto(y,u(y))\). These functions can be chosen so that \([\nu_k]_{C^{0,1}(\R^m;\R^N)}\leq C(M)\) and \(\nu_k(0)=e_{m+k}\). Hence, if
\begin{align}
    \iota_{(y,u(y))}(\theta)=\sum_{k=1}^n\theta_k\nu_k(y), \qquad \theta\in\Sph^{n-1}, \label{iuKhv7XF}
\end{align}
then \(\iota_0(\theta)=(0,\theta)\).

Choose \(r_0\) with \(0<r_0<\reach(M)\). After taking \(\delta>0\) smaller if necessary, we may assume that \(2\delta<r_0\) and
\begin{align*}
    [\pi_M]_{C^{0,1}(M_{2\delta})}\leq \frac{r_0}{r_0-2\delta}\leq2,
\end{align*}
where \(\pi_M\) denotes the nearest-point projection onto \(M\), see \cite[Theorem 4.8]{federer_curvature_1959}. Consider the map \(\psi:M_{2\delta}\to\R^N\times\R^N\) defined by
\begin{align*}
    \psi(z)=(\pi_M(z),z-\pi_M(z)).
\end{align*}
Then, for all \(z,\tilde z\in M_{2\delta}\),
\begin{align*}
    \vert\psi(z)-\psi(\tilde z)\vert^2
    &\leq \big([\pi_M]_{C^{0,1}(M_{2\delta})}^2+(1+[\pi_M]_{C^{0,1}(M_{2\delta})})^2\big)\vert z-\tilde z\vert^2
    \leq16\vert z-\tilde z\vert^2.
\end{align*}
In particular,
\begin{align}
    \vert\psi(z)-\psi(\tilde z)\vert\leq4\vert z-\tilde z\vert \qquad\text{for all }z,\tilde z\in M_{2\delta}. \label{1ZSP6tEk}
\end{align}

We split \(I_j^{(\rho)}\) into a local and a nonlocal part,
\begin{align*}
    I_j^{(\rho)}=J_j^{(\rho)}+\widetilde J_j^{(\rho)},
\end{align*}
where, using the local normal trivialisation above,
\begin{align*}
    J_j^{(\rho)}
    &=\int_{U_\varepsilon}\int_{\Sph^{n-1}}
    \frac{c_j((y,u(y)),\iota_{(y,u(y))}(\theta))J(y)}{\vert-(y,u(y))+\rho t(0,\omega)-\rho\tau\iota_{(y,u(y))}(\theta)\vert^{N+s}}
    \dd\H^{n-1}_\theta\dd\H^m_y,
\end{align*}
and \(\widetilde J_j^{(\rho)}\) is the corresponding integral over \(M\setminus B_\varepsilon\). Here
\begin{align*}
    J(y)=\sqrt{\det(I_m+(\nabla u)^T\nabla u)}
\end{align*}
is the Jacobian of the graph parametrisation.

We claim that, if \(0<R<\delta\), then
\begin{align}
\frac{1}{\vert-(y,u(y))+\rho t(0,\omega)-\rho\tau\iota_{(y,u(y))}(\theta)\vert^{N+s}}
\leq\frac{C}{\big(\vert y\vert^2+\rho^2\vert t\omega-\tau\theta\vert^2\big)^{\frac{N+s}{2}}}. \label{rbUDtvCD}
\end{align}
Indeed, set \(z=\rho t(0,\omega)\) and \(\tilde z=(y,u(y))+\rho\tau\iota_{(y,u(y))}(\theta)\). The assumptions imply that \(z,\tilde z\in M_{2\delta}\), and
\begin{align*}
    \psi(z)&=(0,\rho t(0,\omega)),\\
    \psi(\tilde z)&=((y,u(y)),\rho\tau\iota_{(y,u(y))}(\theta)).
\end{align*}
Thus, by~\eqref{1ZSP6tEk},
\begin{align}
    4\vert-(y,u(y))+\rho t(0,\omega)-\rho\tau\iota_{(y,u(y))}(\theta)\vert
    &\geq\sqrt{\vert(y,u(y))\vert^2+\rho^2\vert t(0,\omega)-\tau\iota_{(y,u(y))}(\theta)\vert^2}. \label{G2uT2ZXv}
\end{align}
Moreover, recalling~\eqref{iuKhv7XF},
\begin{align*}
    t(0,\omega)-\tau\iota_{(y,u(y))}(\theta)
    &=(0,t\omega-\tau\theta)+\tau\sum_{k=1}^n\theta_k(e_{m+k}-\nu_k(y)).
\end{align*}
Consequently, using the Lipschitz bounds on \(\nu_k\) and Young's inequality,
\begin{align*}
    \vert t(0,\omega)-\tau\iota_{(y,u(y))}(\theta)\vert^2
    &\geq\frac12\vert t\omega-\tau\theta\vert^2-C\vert y\vert^2.
\end{align*}
Hence, after taking \(\delta\) smaller if necessary,
\begin{align}
\vert y\vert^2+\rho^2\vert t(0,\omega)-\tau\iota_{(y,u(y))}(\theta)\vert^2
&\geq C\big(\vert y\vert^2+\rho^2\vert t\omega-\tau\theta\vert^2\big). \label{IJkD3qGG}
\end{align}
Combining~\eqref{G2uT2ZXv} and~\eqref{IJkD3qGG} gives~\eqref{rbUDtvCD}.

Using~\eqref{rbUDtvCD}, the boundedness of \(c_j\), and the bound \(\vert J\vert\leq C(M)\), we obtain
\begin{align*}
    \vert J_j^{(\rho)}\vert
    &\leq C\int_{\Sph^{n-1}}\int_{\R^m}\frac{\dd\H^m_y\dd\H^{n-1}_\theta}{\big(\vert y\vert^2+\rho^2\vert t\omega-\tau\theta\vert^2\big)^{\frac{N+s}{2}}}\\
    &\leq \frac{C}{\rho^{n+s}}\int_{\Sph^{n-1}}\frac{\dd\H^{n-1}_\theta}{\vert t\omega-\tau\theta\vert^{n+s}}.
\end{align*}
Since the identification \(\theta\mapsto(0,\theta)\) is an isometry from \(\Sph^{n-1}\) onto \(N_0^1M\), this is the first term on the right-hand side of~\eqref{7qouPfas}.

For \(\widetilde J_j^{(\rho)}\), choose also \(R<\varepsilon/2\) and \(\delta<\varepsilon/4\). If \(y\in M\setminus B_\varepsilon\), then
\begin{align*}
    \vert-y+\rho t(0,\omega)-\rho\tau\eta\vert\geq\vert y\vert-\rho t-\rho\tau\geq\varepsilon-R-\rho\geq\frac14\varepsilon,
\end{align*}
for every \(\eta\in N_y^1M\). Since \(M\) is compact and \(c_j\in L^\infty(N^1M)\), it follows that
\begin{align}
    \vert\widetilde J_j^{(\rho)}\vert\leq C. \label{W4jTZAat}
\end{align}
This proves~\eqref{7qouPfas}.

Finally, we prove~\eqref{QylvaPee}. Since \(c_0=1\), the change of variables \(y\mapsto\rho y\) gives
\begin{align*}
    \rho^{n+s}I_0^{(\rho)}
    &=\int_{U_\varepsilon/\rho}\int_{\Sph^{n-1}}
    \frac{J(\rho y)}{\vert-(y,\rho^{-1}u(\rho y))+t(0,\omega)-\tau\iota_{(\rho y,u(\rho y))}(\theta)\vert^{N+s}}
    \dd\H^{n-1}_\theta\dd\H^m_y+O(\rho^{n+s}).
\end{align*}
Rescaling the pointwise estimate~\eqref{rbUDtvCD}, we have
\begin{align*}
    &\frac{\vert J(\rho y)\chi_{U_\varepsilon/\rho}(y)\vert}{\vert-(y,\rho^{-1}u(\rho y))+t(0,\omega)-\tau\iota_{(\rho y,u(\rho y))}(\theta)\vert^{N+s}}\\
    &\hspace{8em}\leq\frac{C}{\big(\vert y\vert^2+\vert t\omega-\tau\theta\vert^2\big)^{\frac{N+s}{2}}}.
\end{align*}
For \((t,\tau)\neq(1,1)\), the right-hand side belongs to \(L^1(\R^m\times\Sph^{n-1})\). Moreover, for every fixed \(y\),
\begin{align*}
    \rho^{-1}u(\rho y)\to0, \qquad J(\rho y)\to1, \qquad \iota_{(\rho y,u(\rho y))}(\theta)\to(0,\theta).
\end{align*}
Therefore, the dominated convergence theorem gives
\begin{align}
    \lim_{\rho\to0^+}\rho^{n+s}I_0^{(\rho)}
    &=\int_{\Sph^{n-1}}\int_{\R^m}\frac{\dd\H^m_y\dd\H^{n-1}_\theta}{\big(\vert y\vert^2+\vert t\omega-\tau\theta\vert^2\big)^{\frac{N+s}{2}}}
    =\frac{\gamma(n,s)}{\gamma(N,s)}\int_{\Sph^{n-1}}\frac{\dd\H^{n-1}_\theta}{\vert t\omega-\tau\theta\vert^{n+s}}, \label{eq:limit-I0}
\end{align}
where in the last line we used
\begin{align*}
    \int_{\R^m}\frac{\dd\H^m_y}{(\vert y\vert^2+a^2)^{\frac{N+s}{2}}}
    =\frac{\gamma(n,s)}{\gamma(N,s)}a^{-n-s}, \qquad a>0.
\end{align*}
Identifying \(N_0^1M\) with \(\Sph^{n-1}\) proves~\eqref{QylvaPee}.
\end{proof}

We now prove \thref{th:main}. 

\begin{proof}[Proof of \thref{th:main}]
Let \(R,\delta>0\) be as in \thref{HmhFO2hM}. After possibly taking them smaller, assume that \(0<\delta<R<\reach(M)\). Then, for all \(0<\rho<\delta\),
\begin{align*}
\gamma(N,s)^{-1}\Per_s^N(M_\rho)
&=\int_{\R^N\setminus M_\rho}\int_{M_\rho}\frac{\dd y\dd x}{\vert x-y\vert^{N+s}}
=I_1(\rho)+I_2(\rho),
\end{align*}
where
\begin{align*}
I_1(\rho)&=\int_{\R^N\setminus M_R}\int_{M_\rho}\frac{\dd y\dd x}{\vert x-y\vert^{N+s}},\\
I_2(\rho)&=\int_{M_R\setminus M_\rho}\int_{M_\rho}\frac{\dd y\dd x}{\vert x-y\vert^{N+s}}.
\end{align*}

For \(I_1(\rho)\), we claim that
\begin{align}
I_1(\rho)\leq C\rho^n \qquad\text{for all }0<\rho<\delta. \label{ybvOyJ8W}
\end{align}
Indeed, \thref{4gkQqROo} with \(h=1\) gives \(\vert M_\rho\vert\leq C\rho^n\). Let \(\widetilde R>0\) be such that \(M\subset B_{\widetilde R/2}\). Then
\begin{align}
    I_1(\rho)
    &=\int_{B_{\widetilde R}\setminus M_R}\int_{M_\rho}\frac{\dd y\dd x}{\vert x-y\vert^{N+s}}
    +\int_{(\R^N\setminus M_R)\setminus B_{\widetilde R}}\int_{M_\rho}\frac{\dd y\dd x}{\vert x-y\vert^{N+s}}. \label{tJoywSjD}
\end{align}
Since \(\dist(\R^N\setminus M_R,M_\rho)=R-\rho\), the first integral is bounded by \(C\rho^n\). After taking \(\delta<\widetilde R/4\), we have \(M_\rho\subset B_{3\widetilde R/4}\), and therefore the second integral is also bounded by
\begin{align*}
    C\vert M_\rho\vert\int_{\R^N\setminus B_{\widetilde R}}\frac{\dd x}{\vert x\vert^{N+s}}\leq C\rho^n.
\end{align*}
This proves~\eqref{ybvOyJ8W}; in particular, \(\rho^{s-n}I_1(\rho)\to0\).

For \(I_2(\rho)\), by \thref{4gkQqROo},
\begin{align*}
I_2(\rho)
&=\sum_{i,j=0}^m\int_\rho^R\int_M\int_{N_x^1M}\int_0^\rho\int_M\int_{N_y^1M}
\frac{t^{n-1+i}\tau^{n-1+j}c_i(x,\nu)c_j(y,\eta)}{\vert x-y+t\nu-\tau\eta\vert^{N+s}}\\
&\hspace{15em}\dd\H^{n-1}_\eta\dd\H^m_y\dd\tau\dd\H^{n-1}_\nu\dd\H^m_x\dd t.
\end{align*}
The changes of variables \(t\mapsto\rho t\) and \(\tau\mapsto\rho\tau\) yield
\begin{align*}
I_2(\rho)=\sum_{i,j=0}^m\rho^{2n+i+j}A_{ij}^{(\rho)},
\end{align*}
where
\begin{align*}
A_{ij}^{(\rho)}
&=\int_1^{R/\rho}\int_M\int_{N_x^1M}\int_0^1 t^{n-1+i}\tau^{n-1+j}c_i(x,\nu)I_j^{(\rho)}(x,\nu,t,\tau)\\
&\hspace{16em}\dd\tau\dd\H^{n-1}_\nu\dd\H^m_x\dd t.
\end{align*}
By~\eqref{7qouPfas} and the boundedness of \(c_i\),
\begin{align*}
\vert A_{ij}^{(\rho)}\vert
&\leq C\int_1^{R/\rho}\int_M\int_{N_x^1M}\int_0^1 t^{n-1+i}\tau^{n-1+j}
\bigg(\frac1{\rho^{n+s}}\int_{N_x^1M}\frac{\dd\H^{n-1}_\eta}{\vert t\nu-\tau\eta\vert^{n+s}}+1\bigg)\\
&\hspace{16em}\dd\tau\dd\H^{n-1}_\nu\dd\H^m_x\dd t.
\end{align*}
For each \(x\in M\), the fibre \(N_x^1M\) is isometric to \(\Sph^{n-1}\). Hence, by polar coordinates,
\begin{align*}
&\int_1^{R/\rho}\int_{N_x^1M}\int_0^1\int_{N_x^1M}
\frac{t^{n-1+i}\tau^{n-1+j}}{\rho^{n+s}\vert t\nu-\tau\eta\vert^{n+s}}
\dd\H^{n-1}_\eta\dd\tau\dd\H^{n-1}_\nu\dd t\\
&\leq\frac{R^i}{\rho^{n+s+i}}
\int_{\R^n\setminus B_1^n}\int_{B_1^n}\frac{\dd\H^n_y\dd\H^n_x}{\vert x-y\vert^{n+s}}
\leq C\rho^{-n-s-i}.
\end{align*}
Moreover,
\begin{align*}
\int_1^{R/\rho}\int_{N_x^1M}\int_0^1 t^{n-1+i}\tau^{n-1+j}\dd\tau\dd\H^{n-1}_\nu\dd t\leq C\rho^{-n-i}.
\end{align*}
Therefore,
\begin{align*}
    \vert A_{ij}^{(\rho)}\vert\leq C\rho^{-n-s-i}.
\end{align*}
It follows that
\begin{align}
\rho^{s-n}I_2(\rho)
=\sum_{i,j=0}^m\rho^{n+s+i+j}A_{ij}^{(\rho)}
=\sum_{i=0}^m\rho^{n+s+i}A_{i0}^{(\rho)}+O(\rho). \label{Li7JuxU0}
\end{align}

Since \(c_0=1\),
\begin{align*}
\rho^{n+s+i}A_{i0}^{(\rho)}
&=\int_1^{+\infty}\int_M\int_{N_x^1M}\int_0^1
\rho^{n+s+i}\chi_{[1,R/\rho]}(t)c_i(x,\nu)t^{n-1+i}\tau^{n-1}I_0^{(\rho)}(x,\nu,t,\tau)\\
&\hspace{17em}\dd\tau\dd\H^{n-1}_\nu\dd\H^m_x\dd t.
\end{align*}
By~\eqref{7qouPfas}, the absolute value of the integrand is bounded by
\begin{align*}
C\bigg(&R^i t^{n-1}\tau^{n-1}\int_{N_x^1M}\frac{\dd\H^{n-1}_\eta}{\vert t\nu-\tau\eta\vert^{n+s}}
+R^{n+s+i}t^{-1-s}\tau^{n-1}\bigg).
\end{align*}
This upper bound is integrable over
\begin{align*}
\Omega=\{(x,\nu,t,\tau):x\in M,\ \nu\in N_x^1M,\ t\in[1,+\infty),\ \tau\in[0,1]\},
\end{align*}
because its first term is a constant multiple of the fractional perimeter integral of \(B_1^n\) (see the computation below) and its second term is elementary.

Furthermore, by~\eqref{QylvaPee}, for almost every \((x,\nu,t,\tau)\in\Omega\),
\begin{align*}
&\lim_{\rho\to0^+}\rho^{n+s+i}\chi_{[1,R/\rho]}(t)c_i(x,\nu)t^{n-1+i}\tau^{n-1}I_0^{(\rho)}(x,\nu,t,\tau)\\
&\qquad=\frac{\gamma(n,s)}{\gamma(N,s)}\delta_{i0}t^{n-1}\tau^{n-1}
\int_{N_x^1M}\frac{\dd\H^{n-1}_\eta}{\vert t\nu-\tau\eta\vert^{n+s}}.
\end{align*}
Therefore, sending \(\rho\to0^+\) in~\eqref{Li7JuxU0} and applying the dominated convergence theorem, we obtain
\begin{align*}
\lim_{\rho\to0^+}\rho^{s-n}I_2(\rho)
&=\frac{\gamma(n,s)}{\gamma(N,s)}\int_M\int_{N_x^1M}\int_1^{+\infty}\int_0^1\int_{N_x^1M}
\frac{t^{n-1}\tau^{n-1}}{\vert t\nu-\tau\eta\vert^{n+s}}\\
&\hspace{16em}\dd\H^{n-1}_\eta\dd\tau\dd t\dd\H^{n-1}_\nu\dd\H^m_x\\
&=\frac{\gamma(n,s)}{\gamma(N,s)}\H^m(M)
\int_{\R^n\setminus B_1^n}\int_{B_1^n}\frac{\dd\H^n_y\dd\H^n_x}{\vert x-y\vert^{n+s}}\\
&=\gamma(N,s)^{-1}\Per_s^n(B_1^n)\H^m(M).
\end{align*}
Together with \(\rho^{s-n}I_1(\rho)\to0\), this proves \thref{th:main}.
\end{proof}

\section{Weyl tube formula for point sets} \label{sec:points}

In this section, we obtain a full nonlocal tube formula for finite point sets in $\R^N$, first for the fractional kernel and then for general analytic kernels. We use the standard notation for generalised hypergeometric functions: 
\begin{align} \label{eq:hypergeometric}
    {}_pF_q\!\bigg (\begin{matrix}a_1,\ldots,a_p\\b_1,\ldots,b_q\end{matrix};z\bigg)
    :=\sum_{k=0}^\infty
    \frac{(a_1)_k\cdots(a_p)_k}{(b_1)_k\cdots(b_q)_k}\frac{z^k}{k!}, \qquad \text{for }\vert z \vert <1 
\end{align}
where \((a)_k = \frac{\Gamma(a+k)}{\Gamma(a)}\) is the Pochhammer symbol. To establish \thref{cor:pointsets}, we need two intermediate results. The first gives a spherical integration formula for a function of one coordinate.

\begin{lem} \thlabel{lem:coarea}
    Let \(N\geq2\) and \(h:\R \to \R\) be a non-negative measurable function. Then \begin{align*}
        \int_{\partial B_1} h(x_1) \dd \H^{N-1}_x &=  \frac{2^{N-2} \H^{N-1} \big (\partial B_1\big )}{\mathrm B\big (\frac{N-1}2 , \frac12 \big)}\int_0^1 h(1-2t) t^{ \frac{N-3}2   }(1-t)^{ \frac{N-3}2   }\dd t
    \end{align*} where \(\mathrm B\) denotes the Beta function.
\end{lem}

\begin{proof}
    Let \(u:\R^N\to \R\) be defined by \(u(x)=x_1\). Then, on \(\partial B_1\), \(\vert \nabla_{\partial B_1}u(x)\vert  = \sqrt{1-x_1^2}\), so the co-area formula on \(\partial B_1\) yields \begin{align*}
        \int_{\partial B_1} h(x_1) \dd \H^{N-1}_x &= \int_{-1}^1 \int_{\partial B_1 \cap \{x_1=t\} } \frac{h(x_1)}{\sqrt{1-x_1^2}} \dd \H^{N-2}_x \dd t \\
        &= \int_{-1}^1 \frac{h(t)}{\sqrt{1-t^2}}  \H^{N-2} \big (\partial B_1 \cap \{x_1=t\}  \big ) \dd t .
    \end{align*} Since \(\partial B_1 \cap \{x_1=t\}  = \{t\} \times \partial B^{N-1}_{\sqrt{1-t^2}}\), we obtain \begin{align*}
        \int_{\partial B_1} h(x_1) \dd \H^{N-1}_x &=  \H^{N-2} \big (\partial B_1^{N-1}\big )\int_{-1}^1 h(t) \big(1-t^2 \big )^{ \frac{N-3}2   }\dd t .
    \end{align*} Finally, the change of variable \(t\to 1-2t\) yields \begin{align*}
        \int_{\partial B_1} h(x_1) \dd \H^{N-1}_x &=  2^{N-2} \H^{N-2} \big (\partial B_1^{N-1}\big )\int_0^1 h(1-2t) t^{ \frac{N-3}2   }(1-t)^{ \frac{N-3}2   }\dd t \\
        &= \frac{2^{N-2} \H^{N-1} \big (\partial B_1\big )}{\mathrm B\big (\frac{N-1}2 , \frac12 \big)}\int_0^1 h(1-2t) t^{ \frac{N-3}2   }(1-t)^{ \frac{N-3}2   }\dd t
    \end{align*} as claimed, where the final equality follows from the identity \begin{align*}
        \mathrm B\bigg (\frac{N-1}2 , \frac12 \bigg) = \frac{\Gamma \big ( \frac {N-1} 2 \big )\pi^{1/2 }}{\Gamma \big ( \frac N2\big )} = \frac{2 \pi^{ N /2 }}{\Gamma \big ( \frac N2\big )} \cdot \frac{\Gamma \big ( \frac {N-1} 2 \big )}{ 2 \pi^{ \frac{N-1}2  } } = \frac{\H^{N-1} \big (\partial B_1\big )}{\H^{N-2} \big (\partial B_1^{N-1}\big )}.
    \end{align*}
\end{proof}

Our second intermediate result computes the nonlocal interaction between two disjoint balls. 

\begin{prop} \thlabel{prop:interaction-balls} Let \(\lambda > 2\) and \(\alpha \in \R\). Then \begin{align*}
    \int_{B_1}\int_{B_1(\lambda e_1)} \frac{\dd y \dd x }{\vert x-y\vert^\alpha } &= \big ( \H^N(B_1)\big )^2 \lambda^{-\alpha} \cdot  {}_3 F_2 \bigg (
\begin{matrix}
\frac \alpha 2 , \frac{\alpha-N+2}2, \frac{N+1} 2\\
\frac N 2 +1 ,N+1 
\end{matrix}; \frac 4 {\lambda^2}
\bigg ).
\end{align*}
\end{prop}

\begin{proof}
    Here and throughout this proof, we use the notation \(\vert E \vert = \H^N(E)\). Let \begin{align*}
        I= \int_{B_1}\int_{B_1(\lambda e_1)} \frac{\dd y \dd x }{\vert x-y\vert^\alpha }.
    \end{align*} By the change of variable \(y=x+\lambda e_1-z\) in the inner integral below, we have \begin{align*}
         I= \int_{B_1}\int_{B_1(\lambda e_1)} \frac{\dd y \dd x }{\vert x-y\vert^\alpha } &=  \int_{B_1} \int_{B_1(x )} \frac{\dd z \dd x }{\vert \lambda e_1-z \vert^\alpha }. 
    \end{align*} Furthermore, using that \begin{align*}
        \{ (x, z) \in \R^N \times \R^N \text{ s.t. } x\in B_1 , z\in B_1 (x) \} = \{ (x, z) \in \R^N \times \R^N \text{ s.t. } z\in B_2 , x\in B_1\cap  B_1 (z) \},
    \end{align*} Tonelli's theorem implies \begin{align*}
        I &=  \int_{B_2} \frac{\vert B_1 \cap B_1(z)\vert }{\vert \lambda e_1-z \vert^\alpha}  \dd z . 
    \end{align*}  Since \(\vert B_1 \cap B_1(z)\vert = \vert B_1 \cap B_1( \vert z\vert e_1 )\vert  \), polar coordinates then imply \begin{align}
        I&=  \int_0^2 r^{N-1} \vert B_1 \cap B_1(r e_1 )\vert \bigg ( \int_{\partial B_1} \frac{\dd \H^{N-1}_\omega  }{\vert \lambda e_1-r \omega  \vert^\alpha}  \bigg )  \dd r \label{eq:I}
    \end{align} Let \begin{align*}
        \mathcal A(r) = \int_{\partial B_1} \frac{\dd \H^{N-1}_\omega  }{\vert \lambda e_1-r \omega  \vert^\alpha}, \qquad r \in (0,2).  
    \end{align*} We claim that \begin{align}
        \mathcal A(r) &= \H^{N-1} (\partial B_1) \lambda^{-\alpha} \cdot   {}_2 F_1 \bigg (\frac \alpha 2, \frac{\alpha-N+2}2 ;  \frac N2 ; \frac{r^2}{\lambda^2} \bigg )  . \label{eq:int-hypgeom}
    \end{align} Indeed, to prove the claim, we write \begin{align*}
        \mathcal A(r) = \int_{\partial B_1} \frac{\dd \H^{N-1}_\omega  }{\big ( \lambda ^2 + r^2 -2\lambda r \omega_1 \big )^ {\frac\alpha  2 }} = \lambda ^{-\alpha } \int_{\partial B_1} \frac{\dd \H^{N-1}_\omega  }{\big ( 1 + q^2 -2 q \omega_1 \big )^ {\frac\alpha  2 }} 
    \end{align*} where \(q=q(r) = r/\lambda \). Then, assuming \(N\geq 2\), \thref{lem:coarea} implies that \begin{align}
         \mathcal A(r)  &=\frac{2^{N-2} \H^{N-1} \big (\partial B_1\big ) \lambda^{-\alpha} }{\mathrm B\big (\frac{N-1}2 , \frac12 \big)}\int_0^1 \frac{  t^{ \frac{N-3}2   }(1-t)^{ \frac{N-3}2   }}{ \big ( (1-q)^2 +4q t \big  )^{\alpha/2}} \dd t . \label{eq:innerint}
    \end{align} Next, since \cite[Equation 15.6.1]{olver_nist_nodate} implies \begin{align}
        \int_0^1 \frac{  t^{ \frac{N-3}2   }(1-t)^{ \frac{N-3}2   }}{ \big ( (1-q)^2 +4q t \big  )^{\alpha/2}} \dd t &= (1-q)^{-\alpha} \frac{\Gamma \big ( \frac{N-1}2 \big )^2  }{\Gamma(N-1)} \cdot {}_2 F_1 \bigg ( \frac \alpha 2 , \frac{N-1}2 ; N-1 ; \frac{-4q}{(1-q)^2} \bigg )  \label{eq:Hypgeom-identity}
    \end{align} and \begin{align*}
        \frac{2^{N-2} \H^{N-1} \big (\partial B_1\big )}{\mathrm B\big (\frac{N-1}2 , \frac12 \big)} \cdot \frac{\Gamma \big ( \frac{N-1}2 \big )^2  }{\Gamma(N-1)} &= \frac{2^{N-1} \pi^{ \frac{N-1}2   } \Gamma \big ( \frac{N-1}2 \big )}{\Gamma(N-1)} = \frac{2 \pi^{N/2}}{\Gamma \big ( \frac N 2 \big  )} = \H^{N-1} (\partial B_1)
    \end{align*} where the second equality follows from the Legendre duplication formula for the Gamma function, Equations~\eqref{eq:innerint} and~\eqref{eq:Hypgeom-identity} imply \begin{align*}
        \mathcal A(r) &=\H^{N-1} (\partial B_1) \lambda^{-\alpha}(1-q)^{-\alpha} {}_2 F_1 \bigg ( \frac \alpha 2 , \frac{N-1}2 ; N-1 ; \frac{-4q}{(1-q)^2} \bigg ) . 
    \end{align*} Furthermore, by the quadratic transfomation (which follows from \cite[Equation 15.8.14]{olver_nist_nodate} and \cite[Equation 15.8.16]{olver_nist_nodate})\footnote{More explicitly, on one hand\cite[Equation 15.8.14]{olver_nist_nodate} gives \begin{align*}
        {}_2 F_1 \bigg (a,b ; 2b ; \frac{-4q}{(1-q)^2} \bigg ) = \bigg (  \frac{1-q}{1+q}\bigg )^a {}_2F_1 \bigg ( \frac a 2 , b - \frac a2 ; b + \frac12 ; \frac{-4q^2}{(1-q^2)^2} \bigg )
    \end{align*} since, if \(z=\frac{-4q}{(1-q)^2} \) then  \begin{align*}
        1-z = 1 - \bigg ( \frac{-4q}{(1-q)^2}\bigg  ) &= \bigg ( \frac{1+q}{1-q} \bigg )^2 \text{ and }\frac{z^2}{4(z-1)} = \frac{16 q^2}{(1-q)^4} \cdot  \frac{-(1-q)^2}{4(1+q)^2} = \frac{-4q^2}{(1-q^2)^2}.
    \end{align*} On the other hand, \cite[Equation 15.8.16]{olver_nist_nodate} gives \begin{align*}
        {}_2 F_1 \bigg (a,a-b+ \frac12 ; b + \frac12 ; q^2 \bigg ) &= \big (1-q^2)^{-a} {}_2 F_1 \bigg ( \frac a 2 , b - \frac a2  ; b+\frac12 ; \frac{-4q^2}{(1-q^2)^2} \bigg ). 
    \end{align*}Equating these identities yields the claimed quadratic transformation. 
} \begin{align*}
        {}_2 F_1 \bigg (a,b ; 2b ; \frac{-4q}{(1-q)^2} \bigg ) = (1-q)^{2a} {}_2 F_1 \bigg (a,a-b+ \frac12 ; b + \frac12 ; q^2 \bigg ),
    \end{align*} we obtain \begin{align*}
        \mathcal A(r) &=\H^{N-1} (\partial B_1) \lambda^{-\alpha} \cdot   {}_2 F_1 \bigg (\frac \alpha 2, \frac{\alpha-N+2}2 ;  \frac N2 ; q^2 \bigg )
    \end{align*} Recalling that \(q=r/\lambda\), we obtain the claim~\eqref{eq:int-hypgeom}. The claim also holds when \(N=1\) since \begin{align*}
        \int_{\partial B_1} \frac{\dd \H^{0}_\omega }{ \big ( 1+q^2 -2q \omega_1 \big  )^{\alpha/2}}  &= \frac1{  ( 1-q )^\alpha}+\frac1{  ( 1+q )^\alpha} = 2 \cdot  {}_2 F_1 \bigg (\frac \alpha 2, \frac{\alpha+1}2 ;  \frac 12 ; q^2 \bigg ) 
    \end{align*} by \cite[Equation 15.4.7]{olver_nist_nodate}.

    Next, we claim that \begin{align}
        \vert B_1 \cap B_1(re_1)\vert = 2 \H^{N-1} \big ( B_1^{N-1} \big ) \int_{r/2}^1  \big (1-\rho^2\big )^{\frac{N-1}2   } \dd \rho. \label{eq:volume}
    \end{align} provided \(r\in (0,2)\). Indeed, using symmetry then the coarea formula (in \(\R^N\)), we have \begin{align*}
        \vert B_1 \cap B_1(re_1)\vert &=2\vert B_1 \cap \{ x_1 > r/2 \}\vert  \\
        &= 2 \int_{r/2}^1 \H^{N-1} ( B_1 \cap \{x_1=\rho\}  ) \dd \rho \\
        &= 2 \H^{N-1} \big ( B_1^{N-1} \big ) \int_{r/2}^1  \big (1-\rho^2\big )^{\frac{N-1}2   } \dd \rho
    \end{align*} as claimed. 

In the final part of the proof we substitute~\eqref{eq:int-hypgeom} and~\eqref{eq:volume} into~\eqref{eq:I} to obtain \begin{align*}
    I &= 2  \H^{N-1} \big ( B_1^{N-1} \big ) \H^{N-1} \big ( \partial B_1 \big )    \lambda^{-\alpha}  \int_0^2\int_{r/2}^1  r^{N-1}\big (1-\rho^2\big )^{\frac{N-1}2   }   {}_2 F_1 \bigg (\frac \alpha 2, \frac{\alpha-N+2}2 ;  \frac N2 ; \frac{r^2}{\lambda^2} \bigg )  \dd \rho \dd r. 
\end{align*} Next, Tonelli's theorem implies \begin{align*}
    I &= 2 \H^{N-1} \big ( B_1^{N-1} \big ) \H^{N-1} \big ( \partial B_1 \big )   \lambda^{-\alpha}  \int_0^1 \big (1-\rho^2\big )^{\frac{N-1}2   } \int_0^{2\rho }  r^{N-1}   {}_2 F_1 \bigg (\frac \alpha 2, \frac{\alpha-N+2}2 ;  \frac N2 ; \frac{r^2}{\lambda^2} \bigg ) \dd r \dd \rho 
\end{align*} then the change of variable \(r=2\rho t^{1/2}\) yields \begin{align*}
    I &= 2^N \H^{N-1} \big ( B_1^{N-1} \big ) \H^{N-1} \big ( \partial B_1 \big )    \lambda^{-\alpha}  \int_0^1 \rho^N\big (1-\rho^2\big )^{\frac{N-1}2   }   \int_0^1  t^{ \frac N 2 -1}    {}_2 F_1 \bigg (\frac \alpha 2, \frac{\alpha-N+2}2 ;  \frac N2 ; \frac{ 4 \rho^2 t }{\lambda^2} \bigg ) \dd t \dd \rho 
\end{align*} Next, \cite[Equation 15.6.8]{olver_nist_nodate} implies \begin{align*}
     \int_0^1  t^{ \frac N 2 -1}    {}_2 F_1 \bigg (\frac \alpha 2, \frac{\alpha-N+2}2 ;  \frac N2 ; \frac{ 4 \rho^2 t }{\lambda^2} \bigg ) \dd t &= \frac2 N {}_2 F_1 \bigg ( \frac \alpha 2, \frac{\alpha-N+2}2 ; \frac N2 +1 ;    \frac{ 4 \rho^2 }{\lambda^2}\bigg ), 
\end{align*} so, also using that  \(\H^{N-1} \big ( \partial B_1 \big )= N \H^N (B_1)\), \begin{align*}
    I &= \frac{2^{N+1}  \H^{N-1} \big ( B_1^{N-1} \big ) \H^{N-1} \big ( \partial B_1 \big )  \lambda^{-\alpha}} N   \int_0^1 \rho^N\big (1-\rho^2\big )^{\frac{N-1}2   } {}_2 F_1 \bigg ( \frac \alpha 2, \frac{\alpha-N+2}2 ; \frac N2 +1 ;    \frac{ 4 \rho^2 }{\lambda^2}\bigg )  \dd \rho .
\end{align*} Finally, we apply the change of variable \(\rho \to \rho^{1/2}\) then \cite[Equation 16.5.2]{olver_nist_nodate} to conclude \begin{align*}
    I &= 2^{N}  \H^{N-1} \big ( B_1^{N-1} \big )\H^N (B_1) \lambda^{-\alpha}   \int_0^1 \rho^{\frac{N-1}2   }\big (1-\rho\big )^{\frac{N-1}2   } {}_2 F_1 \bigg ( \frac \alpha 2, \frac{\alpha-N+2}2 ; \frac N2 +1 ;    \frac{ 4 \rho }{\lambda^2}\bigg )  \dd \rho \\
    &=  \frac{  2^{N}  \H^{N-1} \big ( B_1^{N-1} \big ) \H^N (B_1) \Gamma \big (\frac{N+1} 2 \big )^2 }{\Gamma (N+1)} \cdot \lambda^{-\alpha} \cdot  {}_3 F_2 \bigg (
\begin{matrix}
\frac \alpha 2 , \frac{\alpha-N+2}2, \frac{N+1} 2\\
\frac N 2 +1 ,N+1 
\end{matrix}; \frac 4 {\lambda^2}
\bigg ) \\ 
&=\big (\H^N (B_1) \big )^2\lambda^{-\alpha} \cdot  {}_3 F_2 \bigg (
\begin{matrix}
\frac \alpha 2 , \frac{\alpha-N+2}2, \frac{N+1} 2\\
\frac N 2 +1 ,N+1 
\end{matrix}; \frac 4 {\lambda^2}
\bigg )
\end{align*} where in the final line we used that \( \H^{N-1} \big ( B_1^{N-1} \big ) = \frac{\pi^{ \frac{N-1}2  } }{\Gamma \big ( \frac{N+1}2 \big )} \) and \(\Gamma(N+1)= 2^N \pi^{-1/2}\Gamma \big ( \frac{N+1}2 \big ) \Gamma \big ( \frac N2 +1 \big )\) (by the Legendre duplication formula) to obtain \begin{align}
   2^{N}  \H^{N-1} \big ( B_1^{N-1} \big ) \cdot \frac{\Gamma \big (\frac{N+1} 2 \big )^2 }{\Gamma (N+1)} &= \frac{ \pi^{N/2} }{\Gamma \big ( \frac N2 +1\big )} = \H^N  ( B_1 ) \label{eq:Gamma-calc}
\end{align}
\end{proof}

We now prove \thref{cor:pointsets}. 

\begin{proof}[Proof of \thref{cor:pointsets}]
    For \(0<\rho<\frac 12 \min_{i\neq j}\vert a_i-a_j\vert\), the balls \(B_\rho(a_i)\) are pairwise disjoint and
\begin{align*}
    M_\rho=\bigcup_{i=1}^P\overline{B_\rho(a_i)},
\end{align*} so 
\begin{align*}
    \Per_s^N(M_\rho)
    &= \gamma(N,s) \sum_{i=1}^P \int_{\R^N\setminus M_\rho} \int_{B_\rho (a_i)} \frac{\dd y \dd x }{\vert x-y\vert^{N+s}}\\
    &=\sum_{i=1}^P\Per_s^N(B_\rho(a_i))
    -\gamma(N,s) \sum_{\stackrel{i,j=1}{i\neq j}}^P\int_{B_\rho(a_j)}\int_{B_\rho(a_i)}\frac{\dd y \dd x }{\vert x-y\vert^{N+s}}\\
    &=\Per_s^N(B_1)\H^0(M) \rho^{N-s}
    -\gamma(N,s) \sum_{\stackrel{i,j=1}{i\neq j}}^P\int_{B_\rho(a_j)}\int_{B_\rho(a_i)}\frac{\dd y \dd x }{\vert x-y\vert^{N+s}}.
\end{align*} Then, by making the changes of variables \(x\to a_j+\rho x\), \(y\to a_j+\rho y\) and using rotational symmetry, we obtain \begin{align*}
    \int_{B_\rho(a_j)}\int_{B_\rho(a_i)}\frac{\dd y \dd x }{\vert x-y\vert^{N+s}} &= \rho^{N-s}\int_{B_1}\int_{B_1\big (\rho^{-1} (a_i-a_j) \big )}\frac{\dd y \dd x }{\vert x-y\vert^{N+s}}=\rho^{N-s}\int_{B_1}\int_{B_1\big (\rho^{-1} \vert a_i-a_j\vert e_1 \big )}\frac{\dd y \dd x }{\vert x-y\vert^{N+s}}. 
\end{align*} Finally, \thref{prop:interaction-balls} implies \begin{align*}
    \int_{B_\rho(a_j)}\int_{B_\rho(a_i)}\frac{\dd y \dd x }{\vert x-y\vert^{N+s}} &=\rho^{2N}\big ( \H^N(B_1)\big )^2 \vert a_i-a_j\vert ^{-N-s} \cdot  {}_3 F_2 \bigg (
\begin{matrix}
\frac {N+s} 2 , \frac{s}2+1, \frac{N+1} 2\\
\frac N 2 +1 ,N+1 
\end{matrix}; \frac {4\rho^2} {\vert a_i-a_j\vert ^2}
\bigg )
\end{align*} completing the proof. 
\end{proof}

Next, we obtain an analogous result to \thref{cor:pointsets} for a more general class of kernels. To this end, given \(K:\R^N \to \R\) measurable, we define \begin{align*}
\Per_K(E) = \int_{\R^N\setminus E}\int_E K(x-y)\dd y \dd x
\end{align*} provided this quantity is finite. In this case, we will not be able to express the tube formula in terms of a hypergeometric function, but instead we express it as a Taylor series expansion in \(\rho\).

\begin{prop} \thlabel{CC0og8h1}
    Let \(P\geq 2\) be an integer and \(M=\{a_1,\dots,a_P\}\) for pairwise distinct points \(a_1,\dots,a_P\in\R^N\). If \(K:\R^N\to\R\) is measurable, \(\Per_K(B)\) is finite for every ball \(B\subset\R^N\), and \(K\in C^\omega(\R^N\setminus\{0\})\), then there exists \(\rho_0>0\), depending on \(M\) and \(K\), such that
    \begin{align*}
        \Per_K(M_\rho)
        &=\Per_K(B_\rho)\H^0(M)-\sum_{\stackrel{i,j=1}{i\neq j}}^P\sum_{\alpha\in\N_0^N}c_\alpha\partial^{2\alpha}K(a_j-a_i)\rho^{2N+2\vert\alpha\vert}
    \end{align*}
    for all \(0<\rho<\rho_0\), where
    \begin{align}
        c_\alpha=\big ( \H^N(B_1) \big )^2  \cdot \frac{\big ( \frac{N+1}2 \big )_{\vert \alpha \vert }  }{ (N+1)_{\vert \alpha \vert } \cdot \big ( \frac N 2 +1\big )_{\vert \alpha \vert} \cdot \alpha! } .  \label{eq:c-alpha}
    \end{align}
\end{prop}

\begin{proof}
For every \(i\neq j\), choose \(r_{ij}>0\) such that the Taylor series of \(K\) at \(a_j-a_i\) converges absolutely and uniformly on compact subsets of \(B_{r_{ij}}(a_j-a_i)\). Choose \(\rho_0>0\) so that
\begin{align*}
    2\rho_0<\min_{i\neq j}\min\{\vert a_i-a_j\vert,r_{ij}\}.
\end{align*}
Then, for \(0<\rho<\rho_0\), the balls \(B_\rho(a_i)\) are pairwise disjoint and
\begin{align*}
    M_\rho=\bigcup_{i=1}^P\overline{B_\rho(a_i)}.
\end{align*}
Consequently,
\begin{align*}
    \Per_K(M_\rho)
    &=\sum_{i=1}^P\Per_K(B_\rho(a_i))
    -\sum_{\stackrel{i,j=1}{i\neq j}}^P\int_{B_\rho(a_j)}\int_{B_\rho(a_i)}K(x-y)\dd y\dd x\\
    &=\Per_K(B_\rho)\H^0(M)
    -\sum_{\stackrel{i,j=1}{i\neq j}}^P\int_{B_\rho(a_j)}\int_{B_\rho(a_i)}K(x-y)\dd y\dd x.
\end{align*}
For \(x\in B_\rho(a_j)\) and \(y\in B_\rho(a_i)\), we have
\begin{align*}
    x-y=(a_j-a_i)+\big(x-y-(a_j-a_i)\big),
\end{align*}
and \(\vert x-y-(a_j-a_i)\vert<2\rho<r_{ij}\). Thus the Taylor series at \(a_j-a_i\) may be integrated term by term:
\begin{align*}
\int_{B_\rho(a_j)}\int_{B_\rho(a_i)}K(x-y)\dd y\dd x
&=\sum_{\alpha\in\N_0^N}\frac{\partial^\alpha K(a_j-a_i)}{\alpha!}
\int_{B_\rho(a_j)}\int_{B_\rho(a_i)}\big(x-y-(a_j-a_i)\big)^\alpha\dd y\dd x.
\end{align*}
The changes of variables \(x=a_j+\rho X\) and \(y=a_i+\rho Y\) give
\begin{align*}
\int_{B_\rho(a_j)}\int_{B_\rho(a_i)}K(x-y)\dd y\dd x
&=\sum_{\alpha\in\N_0^N}\frac{\partial^\alpha K(a_j-a_i)}{\alpha!}\rho^{2N+\vert\alpha\vert}
\int_{B_1}\int_{B_1}(X-Y)^\alpha\dd Y\dd X.
\end{align*}
If any entry of \(\alpha\) is odd then the last integral vanishes by symmetry. Therefore,
\begin{align*}
\int_{B_\rho(a_j)}\int_{B_\rho(a_i)}K(x-y)\dd y\dd x
&=\sum_{\alpha\in\N_0^N}c_\alpha\partial^{2\alpha}K(a_j-a_i)\rho^{2N+2\vert\alpha\vert}
\end{align*} where \(c_\alpha\) is defined by \begin{align*}
     c_\alpha=\frac1{(2\alpha)!}\int_{B_1}\int_{B_1}(x-y)^{2\alpha}\dd y\dd x.
\end{align*} Finally, we show that \(c_\alpha\) has the form~\eqref{eq:c-alpha}. As at the start of the proof of~\thref{prop:interaction-balls}, we make the change of variable \(y=x-z\) in the inner integral, swap the order of integration, then use polar coordinates to obtain \begin{align*}
    c_\alpha &= \frac1{(2\alpha)!}\int_{B_1}\int_{B_1(x)} z^{2\alpha}\dd z \dd x \\
    &= \frac1{(2\alpha)!} \int_{B_2} \vert B_1 \cap B_1(z)\vert  z^{2\alpha}\dd z \\
    &= \frac1{(2\alpha)!} \bigg ( \int_0^2 r^{N+2\vert \alpha\vert  -1} \vert B_1 \cap B_1(r e_1 )\vert \dd r \bigg ) \bigg (  \int_{\partial B_1 }  \omega ^{2\alpha}\dd \H^{N-1}_\omega \bigg ) .
\end{align*} From \cite[Equation 12]{weyl_volume_1939}, we have \begin{align*}
    \int_{\partial B_1 }  \omega ^{2\alpha}\dd \H^{N-1}_\omega  &= \frac{\H^{N-1}(\partial B_1) \prod_{\alpha_k\neq 0, k=1}^N (2\alpha_k - 1)!! }{(N+2\vert \alpha \vert -2) \cdots (N+2) N}. 
\end{align*} Then, using that \((2m-1)!! = \frac{(2m)!}{2^m m!}\), we have \begin{align*}
    \prod_{\alpha_k\neq 0, k=1}^N (2\alpha_k - 1)!! = \prod_{\alpha_k\neq 0, k=1}^N \frac{(2\alpha_k )!}{2^{\alpha_k} (\alpha_k)!} =\prod_{k=1}^N \frac{(2\alpha_k )!}{2^{\alpha_k} (\alpha_k)!} = \frac{(2\alpha)!}{2^{\vert \alpha \vert }\alpha! }
\end{align*} and \begin{align*}
    (N+2\vert \alpha \vert -2) \cdots (N+2) N = \frac{2^{\vert \alpha \vert } N }{N+ 2 \vert \alpha \vert } \cdot \bigg ( \frac N2 +\vert \alpha \vert \bigg  ) \cdots \bigg ( \frac N2 +1 \bigg ) = \frac{2^{\vert \alpha \vert } N }{N+ 2 \vert \alpha \vert } \bigg ( \frac N 2 +1\bigg )_{\vert \alpha \vert},
\end{align*} so \begin{align*}
    \int_{\partial B_1 }  \omega ^{2\alpha}\dd \H^{N-1}_\omega  &= \frac{ (N+2\vert \alpha \vert) \H^{N-1}(\partial B_1) (2\alpha)!  }{2^{2\vert \alpha \vert } N  \cdot \big ( \frac N 2 +1\big )_{\vert \alpha \vert}\cdot  \alpha! }= \frac{ (N+2\vert \alpha \vert) \H^N(B_1) (2\alpha)!  }{2^{2\vert \alpha \vert }   \big ( \frac N 2 +1\big )_{\vert \alpha \vert} \cdot \alpha! } . 
\end{align*} Furthermore, from~\eqref{eq:volume}, we have \begin{align*}
   \int_0^2 r^{N+2 \vert \alpha \vert  -1} \vert B_1 \cap B_1(r e_1 )\vert \dd r &= 2 \H^{N-1}\big ( B_1^{N-1} \big ) \int_0^2 \int_{r/2}^1r^{N+2 \vert \alpha \vert  -1}  (1-t^2)^{\frac{N-1}2} \dd t \dd r \\
   &=2 \H^{N-1}\big ( B_1^{N-1} \big ) \int_0^1 \int_0^{2 t }r^{N+2 \vert \alpha \vert  -1}  (1-t^2)^{\frac{N-1}2} \dd r  \dd t \\
   &= \frac{2^{N+2\vert \alpha \vert +1} \H^{N-1}\big ( B_1^{N-1} \big ) }{N+2\vert \alpha \vert }\int_0^1 t^{N+2 \vert \alpha \vert }  (1-t^2)^{\frac{N-1}2}  \dd t .
\end{align*} Then the change of variable \(t\to t^{1/2}\) gives \begin{align*}
    \int_0^2 r^{N+2 \vert \alpha \vert  -1} \vert B_1 \cap B_1(r e_1 )\vert \dd r &= \frac{2^{N+2\vert \alpha \vert } \H^{N-1}\big ( B_1^{N-1} \big ) }{N+2\vert \alpha \vert }\int_0^1 t^{ \frac{N+2 \vert \alpha \vert-1}2  }  (1-t)^{\frac{N-1}2}  \dd t \\
    &= \frac{2^{N+2\vert \alpha \vert } \H^{N-1}\big ( B_1^{N-1} \big )  \Gamma \big ( \frac{N+2\vert \alpha \vert +1}2 \big )  \Gamma \big ( \frac{N+1}2 \big )  }{(N+2\vert \alpha \vert) \Gamma  (N+\vert \alpha \vert +1  )  }\\
    &= \frac{2^{N+2\vert \alpha \vert } \H^{N-1}\big ( B_1^{N-1} \big )  \big ( \frac{N+1}2 \big )_{\vert \alpha \vert }  \cdot  \Gamma \big ( \frac{N+1}2 \big )  ^2 }{(N+2\vert \alpha \vert) (N+1)_{\vert \alpha \vert }\cdot  \Gamma  (N +1  )  } 
\end{align*} using the identities \(\int_0^1 t^{a-1}(1-t)^{b-1}\dd t = \mathrm B(a,b) = \frac{\Gamma(a)\Gamma(b)}{\Gamma(a+b)}\) and \(\Gamma(x+k)=(x)_k \cdot \Gamma(x)\). Collecting the preceding identities, we obtain \begin{align*}
    c_\alpha &= \frac1{(2\alpha)!} \bigg ( \frac{2^{N+2\vert \alpha \vert } \H^{N-1}\big ( B_1^{N-1} \big )  \big ( \frac{N+1}2 \big )_{\vert \alpha \vert }  \cdot  \Gamma \big ( \frac{N+1}2 \big )  ^2 }{(N+2\vert \alpha \vert) (N+1)_{\vert \alpha \vert }\cdot  \Gamma  (N +1  )  }  \bigg  )\bigg ( \frac{ (N+2\vert \alpha \vert) \H^N(B_1) (2\alpha)!  }{2^{2\vert \alpha \vert }   \big ( \frac N 2 +1\big )_{\vert \alpha \vert} \cdot \alpha! }\bigg  )\\
    &=  \frac{2^{N } \H^{N-1}\big ( B_1^{N-1} \big )\H^N(B_1)     \cdot  \Gamma \big ( \frac{N+1}2 \big )  ^2 }{   \Gamma  (N +1  )  }  \cdot \frac{\big ( \frac{N+1}2 \big )_{\vert \alpha \vert }  }{ (N+1)_{\vert \alpha \vert } \cdot \big ( \frac N 2 +1\big )_{\vert \alpha \vert} \cdot \alpha! }.
\end{align*} Finally, using~\eqref{eq:Gamma-calc}, we obtain~\eqref{eq:c-alpha}, as desired.
\end{proof}

\printbibliography
\vfill
\end{document}